\documentclass[11pt]{article}
\usepackage{fullpage}
\usepackage[english]{babel}
\usepackage{amsthm,amsfonts,amssymb,amsmath,mathrsfs,textcomp,amscd,amsbsy,amsfonts}
\usepackage[colorlinks=true, pdfstartview=FitH, linkcolor=blue, citecolor=blue, urlcolor=blue]{hyperref}

\usepackage{enumitem}
\usepackage{eurosym}
\usepackage{bbm}
\usepackage{mathrsfs}
\usepackage{mathtools,xcolor}
\usepackage{subcaption}
\usepackage{graphicx}
\usepackage{float}
\usepackage[normalem]{ulem}
\usepackage{biblatex}
\usepackage{csquotes}
\usepackage{microtype}

\theoremstyle{plain}
\newtheorem{thm}{Theorem}[section]
\newtheorem{prop}[thm]{Proposition}
\newtheorem{lemma}[thm]{Lemma}

\theoremstyle{definition}
\newtheorem{definition}[thm]{Definition}
\newtheorem{rem}[thm]{Remark}
\numberwithin{equation}{section}

\renewcommand{\mod}[1]{{\ifmmode\text{\rm\ (mod~$#1$)}\else\discretionary{}{}{\hbox{\!\!}}\rm(mod~$#1$)\fi}}

\renewcommand{\Re}{\mathop{\rm Re}}
\newcommand{\erf}{\mathop{\rm erf}}

\renewcommand{\bar}{\overline}

\newcommand{\Edot}{\mathord{\mathring E}}

\newcommand{\R}{\mathbb R}
\newcommand{\C}{\mathbb C}

\newcommand{\p}{\mathcal P}

\newcommand{\Xva}{X}

\title{Joint distributions of error terms for primes in arithmetic progressions modulo~$11$}
\author{K\"{u}bra Benl\.{i}, Greg Martin, and Paul P\'{e}ringuey}

\newcommand{\Addresses}{ \bigskip
  \footnotesize

K.~Benli, \textsc{Department of Mathematics,
Bo\u{g}az\.{i}\c{c}\.{i} University,  
Bebek, \.{I}stanbul 34342 
T\"{u}rk\.{i}ye}\par\nopagebreak
  \textit{E-mail address:}  \texttt{kubra.benli@boun.edu.tr}

  \medskip

  G.~Martin, \textsc{Department of Mathematics,
University of British Columbia,
Room 121, 1984 Mathematics Road,
Vancouver, BC  V6T 1Z2, Canada}\par\nopagebreak
  \textit{E-mail address:} \texttt{gerg@math.ubc.ca}

}

\date{}

\begin{document}

\maketitle

\begin{abstract}
We provide a formula for the logarithmic density of the set of positive real numbers on which two prime counting functions $\psi(x;q,a)$ and $\psi(x;q,b)$ are simultaneously larger than their asymptotic main terms, as well as a method for calculating the numerical values of such densities with rigorously bounded errors. We apply these formulas to the pairwise races in the case $q=11$, determining which pairs of residues~$a$ and~$b$ are more or less correlated in this way. The outcomes when $q=11$ provide a deeper mathematical illumination of the ``mirror image'' and ``cyclic ordering'' phenomena observed by Bays and Hudson.
\end{abstract}

\section{Introduction}

The prime counting functions in arithmetic progressions
\[
\pi(x;q,a) = \#\{ p\le x\colon p\equiv a\mod q\}
\quad\text{and}\quad
\psi(x;q,a) = \sum_{\substack{n\le x\\n\equiv a\mod q}}\Lambda(n)=\sum_{\substack{p^k\le x\\p^k\equiv a\mod q}}\log p
\]
(for integers $q\ge3$ and $(a,q)=1$) are central objects of study in both classical analytic number theory and comparative prime number theory. Both functions have asymptotic formulas, and the normalized error terms
\begin{equation} \label{NETs}
\Edot^\pi(x;q,a) = \frac{\varphi(q)\pi(x;q,\alpha)-\pi(x)}{\sqrt{x}/\log x}
\quad\text{and}\quad
\Edot^\psi(x;q,a) = \frac{\varphi(q)\psi(x;q,\alpha)-\psi(x)}{\sqrt{x}}
\end{equation}
are believed to have limiting logarithmic distribution functions that are symmetric about their mean values. Indeed, Rubinstein and Sarnak~\cite{RS94} established the existence of these distribution functions under two hypotheses on the zeros of Dirichlet $L$-functions:
\begin{itemize}\setlength\itemsep{0pt}
\item the generalized Riemann hypothesis (GRH), which asserts that all nontrivial zeros of $L(s,\chi)$ are of the form $\frac12+i\gamma$ for $\gamma\in\R$;
\item the linear independence hypothesis (LI), which asserts that the collection of all nonnegative such ordinates~$\gamma$ is linearly independent over the rational numbers.
\end{itemize}
In particular, spurred by observations (going back to Chebyshev) that some arithmetic progressions seem to contain more primes than others, Rubinstein and Sarnak's work allowed them to examine comparisons between two prime counting functions (equivalently, between their normalized error terms). For instance, they showed how to calculate the logarithmic density of the set of positive real numbers~$x$ for which $\Edot^\pi(x;q,a_1) > \Edot^\pi(x;q,a_2)$ (which is the same as the set for which $\pi(x;q,a_1) > \pi(x;q,a_2)$, since jointly normalizing the error terms preserves their ordering).

In this paper, we are interested not in direct comparisons of these normalized error terms but rather in correlations between them. Our work is motivated by observations of Bays and Hudson~\cite{BaysHudson} in the case $q=11$, where they plotted the ten normalized error terms $\Edot^\pi(x;11,a)$ for $1\le a\le 10$. (Their plots actually represented normalized error terms corresponding to the integrals $\int_0^{x}\pi(u;11,a)\,du$, but we will consider the unintegrated version in this paper.) Bays and Hudson pointed out two striking phenomena regarding the leaders (the residue class~$a$ that maximizes $\pi(x;11,a)$ for given values of~$x$) and trailers (the residue class~$a$ that minimizes $\pi(x;11,a)$) of this ten-way prime number race.
\begin{itemize}\setlength\itemsep{0pt}
\item {\bf Mirror image phenomenon}: The leader and the trailer tended to be additive inverses of each other modulo~$11$; for example, $\pi(x;11,8)$ tend to be largest at the same time that $\pi(x;11,3)$ was smallest.
\item {\bf Cyclic ordering phenomenon}: The leaders came from the set of quadratic nonresidues $8,6,10,2,7$ while the trailers came from the set of quadratic residues $3,5,1,9,4$ (this much was familiar to those acquainted with comparative prime number theory); surprisingly, however, the leaders and trailers tended to cycle through those lists {\em in that particular order} as~$x$ increased.
\end{itemize}

In theory, one could study the mirror image phenomenon by computing the logarithmic densities of all~$10!$ possible orderings of the $\pi(x;11,a)$, restricting to the~$9!$ orderings where the leader is~$8$ (for instance), calculating the total of the~$8!$ such densities for which~$3$ is simultaneously the trailer, and comparing this total to the corresponding totals for other potential trailers. Unfortunately, this ten-dimensional problem is computationally intractable in practice. Instead, we focus in this paper on pairs of error functions $\Edot^\pi(x;11,a_1)$ and $\Edot^\pi(x;11,a_2)$, calculating how often they are both larger than their mean values. The mirror image phenomenon suggests that $\Edot^\pi(x;11,8)$ and $\Edot^\pi(x;11,3)$ (for instance) are negatively correlated and that the density of the set of~$x$ where both are larger than average should be less than $\frac12\cdot\frac12=\frac14$, the product of the densities of the sets where the two error terms are individually larger than average. (We remark in passing that Fiorilli and the second author~\cite[Section~4.3]{BSRPNR} took a different approach to analyzing the mirror image phenomenon by calculating variances of the logarithmic distributions of sums $\Edot^\pi(x;11,a_1) + \Edot^\pi(x;11,a_2)$.) Since we are interested in comparing these error terms $\Edot^\pi(x;11,a)$ to their mean values (which depend on whether~$a$ is a quadratic residue or nonresidue), we actually consider the related error terms $\Edot^\psi(x;11,a)$, which have the same logarithmic distribution except that all their mean values vanish.

Studying the cyclic ordering phemonenon using the tools of comparative prime number theory seems unlikely, since the limiting logarithmic distribution (essentially a continuous normalized histogram of the values taken by the function) loses all information about the sequence in which those values were taken. Upon reflection, though, it seems reasonable to again use the correlations of pairs of error terms as a proxy for this phenomenon. Consider a real number~$x$ for which~$10$ is the leader, for instance, so that $\pi(x;11,10)$ is largest among the ten prime counting functions: if it is more likely that~$6$ is the previous (distinct) leader and~$2$ is the subsequent leader, then it seems that both $\pi(x;11,6)$ and $\pi(x;11,2)$ should be reasonably large as well, since $\pi(x;11,6)$ was recently the largest of the ten and $\pi(x;11,2)$ soon will be. We postulate, therefore, that $\Edot^\pi(x;11,10)$ is more positively correlated with both $\Edot^\pi(x;11,6)$ and $\Edot^\pi(x;11,2)$ than it is with either $\Edot^\pi(x;11,8)$ or $\Edot^\pi(x;11,7)$.

With this discussion in mind, we define the following correlation densities.

\begin{definition}
Given a modulus $q\ge3$, two distinct reduced residue classes~$a_1$ and~$a_2$ modulo~$q$, and two signs $s_1,s_2\in\{-,+\}$, define $\delta^{s_1,s_2}(q;a_1,a_2)$ to be the logarithmic density of the set of positive real numbers~$x$ for which both $s_1 \Edot^\psi(x;q,a_1) > 0$ and $s_2 \Edot^\psi(x;q,a_2) > 0$. For example, $\delta^{++}(q;a_1,a_2)$ is the logarithmic density of the set where $\Edot^\psi(x;q,a_1)$ and $\Edot^\psi(x;q,a_2)$ are simultaneously positive, or equivalently the logarithmic density of the set where the vector-valued function $\bigl( \Edot^\psi(x;q,a_1), \Edot^\psi(x;q,a_2) \bigr)$ lies in the open first quadrant.

Furthermore, when~$q$ is understood from context, let $\delta^{s_1,s_2}_a$ be a shorthand for $\delta^{s_1,s_2}(q;a,1)$, and indeed for $\delta^{s_1,s_2}(q;ab,b)$ for any reduced residue class $b\mod q$. (We will see that $\delta^{s_1,s_2}(q;ab,b)$ is independent of the choice of~$b$ in Proposition~\ref{proposition ANS} below; this symmetry is also implicit in the work of Feuerverger and the second author~\cite[Theorem~2(b)--(c)]{BSRPNR}, although there they treated $\Edot^\pi(x;q,a)$ rather than $\Edot^\psi(x;q,a)$.)
\end{definition}

We may now state our main result, concerning these correlation densities modulo~$11$.

\begin{thm}\label{mainthmshort}
Assume GRH and LI. When $q=11$, the following tables contain numerical values for the logarithmic densities $\delta_a^{++} = \delta^{++}(11;a,1) = \delta^{++}(11;ab,b)$ for any $1\le b\le10$.
\begin{center}
$\begin{array}{ |c|c| } 
\hline
a \mod{11} & \delta_a^{++} \\  \hline\hline
2 &0.21829017 \\ \hline
3 &0.21355913 \\ \hline
4 &0.21355913 \\ \hline
5 &0.25307193 \\ \hline
6 &0.21829017 \\ \hline
7 &0.26736689 \\ \hline
8 &0.26736689 \\ \hline
9 &0.25307193 \\ \hline
10&0.18561178 \\ \hline
\end{array}$
\qquad\qquad
$\begin{array}{ |c|c| } 
\hline
a \mod{11} & \delta_a^{++} \\  \hline\hline
8^1\equiv8 & 0.26736689 \\ \hline
8^2\equiv9 & 0.25307193 \\ \hline
8^3\equiv6 & 0.21829017 \\ \hline
8^4\equiv4 & 0.21355913 \\ \hline
8^5\equiv10 & 0.18561178 \\ \hline
8^6\equiv3 & 0.21355913 \\ \hline
8^7\equiv2 & 0.21829017 \\ \hline
8^8\equiv5 & 0.25307193 \\ \hline
8^9\equiv7 & 0.26736689 \\ \hline
\end{array}$
\end{center}
(The second table is simply a reordering of the first.)
In all cases, the given values are rigorously proven correct to within $4 \times10^{-8}$. 

Moreover, for any $q\ge3$, we have $\delta_a^{--}=\delta_a^{++}$ and also $\delta_a^{+-}=\delta_a^{-+}=\frac12-\delta_a^{++}$.
\end{thm}

\begin{rem}
The bottom left entry, which is the smallest density in the table, confirms the mirror image phenomenon: it shows that $\Edot^\psi(x;q,a)$ and $\Edot^\psi(x;q,10a) = \Edot^\psi(x;q,11-a)$ are more rarely simultaneously positive than any other pair of error terms.
\end{rem}

\begin{rem}
The rows with even exponents in the right-hand table contain densities relevant to the cyclic ordering phenomenon for trailers, since $\delta_a^{++} = \delta_a^{--}$ measures how often $\Edot^\psi(x;q,a)$ and $\Edot^\psi(x;q,1)$ are simultaneously negative. Note that~$5$ and~$9$ are the neighbours of~$1$ in the cyclic ordering of quadratic residues, and indeed $\delta_5^{--} = \delta_9^{--} \approx 0.253072$ is (slightly) greater than~$\frac14$; in contrast, $\delta_3^{--} = \delta_4^{--} \approx 0.213560$ is rather smaller than~$\frac14$. Those same rows also confirm the cyclic ordering phenomenon for leaders, once we note that $\delta_a^{++} = \delta^{++}(11;10a,10) = \delta^{++}(11;11-a,10)$ for example.

Taken as a whole, the right-hand table suggests an even more sweeping cyclic ordering phenomenon for the race among the $\psi(x;q,a)$ than Bays and Hudson observed for the race among the $\pi(x;q,a)$. In the $\psi(x;q,a)$ race, all ten contestants have mean value~$0$ and therefore are all equally likely to be leaders (or trailers); the data suggests that the leaders have a tendency to cycle through the list $8^1,8^2,\dots,8^9,8^{10} = 8,9,6,4,10,3,2,5,7,1$ as~$x$ increases. In Section~\ref{model section} we describe a model that indicates that the cyclic ordering phenomenon is caused by a single low-lying zero (and its complex conjugate) of a Dirichlet $L$-function with conductor~$11$.
\end{rem}

We assume GRH and LI whenever needed throughout this paper. In the first half of the paper, we consider the densities $\delta^{++}(q;a_1,a_2)$ for general moduli~$q$; we adapt the method of Feuerverger and the second author~\cite{BSRPNR} to write down a formula for these densities in terms of an integral involving a two-dimensional Fourier transform. In the second half, we specialize to the case $q=11$ and describe how to compute these densities with rigorously bounded error terms, and compare them with the aforementioned model.

\section{A formula for the probabilities}

The tools of comparative prime number theory allow us to write down a formula for densities associated with the error terms $\Edot^\psi(x;q,a)$, in terms of integrals involving the Fourier transform of the limiting logarithmic distributions of those error terms. In this section we present these tools in a somewhat different way, involving tempered distributions, which we hope clarifies the process of calculating such formulas. First, however, we begin by writing down that Fourier transform and noting a connection to a particular random variable.

\subsection{A random variable and its characteristic function}

In this section, we set our notation and focus on the two-dimensional random variable relevant to the densities in Theorem~\ref{mainthmshort}. All the material in this section is well known to experts, although the statements from this section are hard to find exactly as written in the literature---primarily because most authors focus on $\pi(x;q,a)$ instead of $\psi(x;q,a)$, and also because we prefer explicitly using random variables rather than having them implicit in the background. The reader can refer to~\cite[Section~1]{ANS},~\cite[Section~2.1]{ISRPNR},~\cite[Section~12.1]{MV}, and~\cite[Section~2.1]{RS94} for closely analogous formulas.

\begin{definition} \label{X def}
Given a modulus $q\ge3$ and a reduced residue class $a\mod q$, define the random variable
\[
\Xva(q;a) = \sum_{\substack{\chi\mod{q}\\\chi\neq\chi_0}} (1,\chi(a)) \sum_{\substack{\gamma\in\R \\ L(1/2+i\gamma,\chi)=0}}\frac{Z_\gamma}{1/2+i\gamma},
\]
where the $Z_\gamma$ are random variables uniformly distributed on the unit circle in~$\C$, independent except that $Z_{-\gamma}=\bar{Z_{\gamma}}$; because of this convention and the symmetries of zeros of Dirichlet $L$-functions, the components of~$X(q;a)$ are actually real-valued. When~$q$ and~$a$ are understood from context, we write~$X$ and~$d\mu_X$ for the random variable~$X(q;a)$ and its distribution, respectively. Also, the characteristic function of~$X(q;a)$ is
\begin{equation} \label{characteristicfunctionX}
\phi_X(t_1,t_2) = \prod_{\substack{\chi\mod{q}\\ \chi\neq \chi_0}} \prod_{\substack{\gamma>0\\ L(1/2+i\gamma,\chi)=0}} J_0\biggl( \frac{2|t_1+\chi(a)t_2|}{\sqrt{1/4+\gamma^2}} \biggr),
\end{equation}
which follows from the standard computation that the characteristic function of $(Z_\gamma+\overline{Z_\gamma})(v_1,v_2)$ is $J_0\bigl( 2|v_1t_1+v_2t_2| \bigr)$ for any $(v_1,v_2),(t_1,t_2)\in\R^2$. It is obvious that $\phi_X(t_1,t_2)=\phi_X(-t_1,-t_2)$; it turns out also to be true that $\phi_X(t_1,t_2)=\phi_X(t_2,t_1)$ when $t_1,t_2\in\R$, since
\[
|t_2+\chi(a)t_1| = \bigl| \chi(a)(\bar\chi(a)t_2+t_1) \bigr| = |t_1+\bar\chi(a)t_2| = |t_1+\chi(a)t_2|.
\]
\end{definition}

\begin{prop} \label{proposition ANS}
Assume GRH and LI. The limiting logarithmic distribution of the vector $\bigl( \Edot^\psi(x;q,\alpha),\Edot^\psi(x;q,\beta) \bigr)$ exists, and the Fourier transform of that distribution is
\begin{equation} \label{characteristicfunctionpsi}
\prod_{\substack{\chi\mod{q}\\ \chi\neq \chi_0}} \prod_{\substack{\gamma>0\\ L(1/2+i\gamma,\chi)=0}} J_0\biggl( \frac{2|t_1+\chi(\alpha\beta^{-1})t_2|}{\sqrt{1/4+\gamma^2}} \biggr),
\end{equation}
where $J_0(z)$ denotes the standard Bessel function of order zero. In particular, the limiting logarithmic distribution of $\bigl( \Edot^\psi(x;q,\alpha),\Edot^\psi(x;q,\beta) \bigr)$ is the same as the distribution of the random variable $\Xva(q;a)$ where $a\equiv\alpha\beta^{-1}\mod q$.
\end{prop}

\begin{proof}
We begin with the explicit formula (assuming GRH)
\begin{align*}
\psi(x,q,a) = 
\frac{\psi(x)}{\varphi(q)} - \frac1{\varphi(q)}\sum_{\substack{\chi\mod{q}\\\chi\neq \chi_0}}\bar{\chi}(a)\sum_{\substack{\gamma\in\R \\ |\gamma|\le T \\ L(1/2+i\gamma,\chi)=0}}\frac{x^{1/2+i\gamma}}{1/2+i\gamma} + E_T(x)
\end{align*}
for $x,T\ge2$,
where $\varphi$ denotes the Euler totient function, and $E_T(x)$ is an error term for which suitable bounds exist both pointwise and in $L^2$-norm. It follows from the definition~\eqref{NETs} that
\begin{align*}
\bigl( \Edot^\psi(x;q,\alpha),\Edot^\psi(x;q,\beta) \bigr) &= - \sum_{\substack{\chi\mod{q}\\\chi\neq \chi_0}} \bigl( \bar{\chi}(\alpha), \bar{\chi}(\beta) \bigr) \sum_{\substack{\gamma\in\R \\ |\gamma|\le T \\ L(1/2+i\gamma,\chi)=0}}\frac{x^{i\gamma}}{1/2+i\gamma} + \frac{E_T(x)}{\sqrt x}.
\end{align*}
This formula allows us to apply~\cite[Theorem~1.9]{ANS} to conclude that the limiting logarithmic distribution of the left-hand side does indeed exist and that its Fourier transform is
\[
\prod_{\substack{\chi\mod{q}\\ \chi\neq \chi_0}} \prod_{\substack{\gamma>0\\ L(1/2+i\gamma,\chi)=0}} J_0\biggl( \frac{2|{-}\bar\chi(\alpha)t_1-\bar\chi(\beta)t_2|}{\sqrt{1/4+\gamma^2}} \biggr).
\]
Because $|{-}\bar\chi(\alpha)t_1-\bar\chi(\beta)t_2| = \bigl| {-}\bar\chi(\alpha) \bigl( t_1+\chi(\alpha\beta^{-1})t_2 \bigr) \bigr| = |t_1+\chi(\alpha\beta^{-1})t_2|$,
this formula is equivalent to equation~\eqref{characteristicfunctionpsi}. The final assertion of the proposition follows upon comparison to equation~\eqref{characteristicfunctionX}. 
\end{proof}

\subsection{Tempered distributions}

In this section, we gather just enough information about tempered distributions for our purposes in the next section.

\begin{definition}
The {\em Schwartz space} $\mathcal{S}(\R^d)$ is the vector space of infinitely differentiable functions $\phi\colon \R^d\to\C$ such that every partial derivative of~$\phi$ decays to zero (in all directions) even when multiplied by any polynomial. In particular,
\begin{equation} \label{seminorm}
\|\phi\|_{a_1,\dots,a_d,b_1,\dots,b_d} = \sup \biggl\{ \biggl| t_1^{a_1} \cdots t_d^{a_d} \frac{\partial^{b_1+\cdots+b_d}}{\partial t_1^{b_1} \cdots \partial t_d^{b_d}} \phi(\vec t) \biggr| \colon \vec t\in\R^d \biggr\}
\end{equation}
exists for all nonnegative integers $a_1,\dots,a_d,b_1,\dots,b_d$.

The space of {\em tempered distributions} is the dual space of Schwartz space; in other words, a tempered distribution is a function (or operator) $T\colon \mathcal{S}(\R^d) \to \C$ that is both linear and continuous.
We write $\langle T,\phi\rangle$ as a synonym for $T(\phi)$.

The {\em Fourier transform} of a tempered distribution~$T$ is the tempered distribution $\mathcal{F}(T)$ such that $\langle \mathcal{F}(T),\phi \rangle = \langle T,\mathcal{F}(\phi) \rangle$ for all $\phi\in\mathcal{S}(\R^d)$, where~$\mathcal{F}(\phi)$ is the usual Fourier transform of an integrable function on~$\R$. This Fourier transform is an automorphism of the space of tempered distributions, and indeed it is straightforward to check that
\begin{equation} \label{F inverse}
\mathcal{F}^{-1}(T) = \frac1{(2\pi)^d} \overline{\mathcal{F}(\bar T)}.
\end{equation}
\end{definition}

\begin{rem}
One example of a tempered distribution is the Dirac delta operator $\delta_d$ which satisfies $\langle \delta_d,\phi\rangle = \phi(\vec 0)$.
Another example is an integral operator $\langle F,\phi\rangle = \int_{\R^d} f(\vec t)\phi(\vec t) \,d\vec t$ where~$f$ is a function of at most polynomial growth.
Indeed we will be primarily interested in this latter tempered distribution with~$f$ being the indicator function $\mathbbm1_{[0,\infty)\times[0,\infty)}$ (because of equation~\eqref{Heaviside2} below); its Fourier transform will involve singular tempered distributions such as~$\delta_2$.
\end{rem}

\begin{prop} \label{singular tempered}
The following operators are all tempered distributions:
\begin{enumerate}[label={\rm(\alph*)},itemsep=0pt]
\item the Dirac $\delta_2\colon\mathcal{S}(\R^2)\rightarrow \R$, which is the tempered distribution such that $\langle \delta_2,\phi \rangle=\phi(0,0)$ for all $\phi\in \mathcal{S}(\R^2)$.
\item the ``left-variable Cauchy principal value'' $PV_\ell\colon\mathcal{S}(\R^2)\rightarrow \R$, which is the tempered distribution such that $\langle PV_\ell,\phi \rangle=\lim_{\varepsilon\rightarrow 0}\int_{|u|\ge \varepsilon}\frac{\phi(u,0)}{u}\,du$ for all $\phi\in \mathcal{S}(\R^2)$.
\item the ``right-variable Cauchy principal value'' $PV_r\colon\mathcal{S}(\R^2)\rightarrow \R$, which is the tempered distribution such that $\langle PV_r,\phi \rangle=\lim_{\varepsilon\rightarrow 0}\int_{|v|\ge \varepsilon}\frac{\phi(0,v)}{v}\,dv$ for all $\phi\in \mathcal{S}(\R^2)$.
\item the two-dimensional Cauchy principal value $PV_2\colon\mathcal{S}(\R^2)\rightarrow \R$, which is the tempered distribution such that $\langle PV_2,\phi \rangle=\lim_{\varepsilon\rightarrow 0}\int_{|u|\ge \varepsilon}\int_{|v|\ge \varepsilon}\frac{\phi(u,v)}{uv}\,dv\,du$ for all $\phi\in \mathcal{S}(\R^2)$.
\end{enumerate}
\end{prop}

\begin{proof}
The nontrivial part of the proof is to show that the limit in part~(d) actually exists; we confirm that fact first and then establish the proposition in its entirety.

Given $\phi\in\mathcal{S}(\R^2)$ and $\varepsilon>0$, we write
\begin{align}
\int_{|u|\ge \varepsilon}\int_{|v|\ge \varepsilon}\frac{\phi(u,v)}{u v}\,dv\,du &= \int_{\varepsilon}^\infty \int_{\varepsilon}^\infty \frac{\phi(u,v)-\phi(u,-v)-\phi(-u,v)+\phi(-u,-v)}{u v}\,dv\,du \notag \\
&= I_0 + I_u(\varepsilon) + I_v(\varepsilon) + I_{uv}(\varepsilon),
\label{Iuv}
\end{align}
where we define
\begin{align*}
I_0 &= \int_1^\infty \int_1^\infty \frac{\phi(u,v)-\phi(u,-v)-\phi(-u,v)+\phi(-u,-v)}{u v}\,dv\,du \\
I_v(\varepsilon) &= \int_1^\infty \int_\varepsilon^1\frac{\phi(u,v)-\phi(u,-v)-\phi(-u,v)+\phi(-u,-v)}{u v}\,dv\,du\\
I_u(\varepsilon) &= \int_\varepsilon^1 \int_1^\infty \frac{\phi(u,v)-\phi(u,-v)-\phi(-u,v)+\phi(-u,-v)}{u v}\,dv\,du\\
I_{uv}(\varepsilon) &= \int_\varepsilon^1\int_\varepsilon^1\frac{\phi(u,v)-\phi(u,-v)-\phi(-u,v)+\phi(-u,-v)}{u v}\,dv\,du.
\end{align*}
Certainly $I_0$ converges, indeed
\begin{align*}
|I_0| &\le \int_1^\infty \int_1^\infty \frac{|uv\phi(u,v)|+|u(-v)\phi(u,-v)|+|(-u)v\phi(-u,v)|+|(-u)(-v)\phi(-u,-v)|}{u^2 v^2}\,dv\,du \\
&\le 4\|\phi\|_{1,1,0,0}
\end{align*}
in the notation of equation~\eqref{seminorm}. For~$I_v(\varepsilon)$, we use the identity $\phi(u,v)-\phi(u,-v)=\int_{-v}^{v}\frac{\partial \phi}{\partial v}(u,t) \,dt$ to write
\begin{align*}
|I_v(\varepsilon)| &= \biggl| \int_1^\infty \int_\varepsilon^1 \int_{-v}^v \frac{u\frac{\partial \phi}{\partial v}(u,t) - u\frac{\partial \phi}{\partial v}(-u,t)}{u^2 v} \,dt\,dv\,du \biggr| \\
&\le \int_1^\infty \int_\varepsilon^1 \int_{-v}^v \frac{2\|\phi\|_{1,0,0,1}}{u^2v} \,dt\,dv\,du \le 4\|\phi\|_{1,0,0,1},
\end{align*}
and in particular~$I_v(\varepsilon)$ converges by the dominated convergence theorem. A similar argument shows that $|I_u(\varepsilon)| \le 4\|\phi\|_{0,1,1,0}$. Finally,
\begin{align*}
|I_{uv}(\varepsilon)| &= \biggl| \int_1^\infty \int_\varepsilon^1 \int_{-v}^v \frac{\frac{\partial \phi}{\partial v}(u,t) - \frac{\partial \phi}{\partial v}\phi(-u,t)}{uv} \,dt\,dv\,du \biggr| \\
&= \biggl| \int_1^\infty \int_\varepsilon^1 \int_{-v}^v \int_{-u}^u \frac{\frac{\partial^2 \phi}{\partial v\partial u}(s,t)}{uv} \,ds\,dt\,dv\,du \biggr| \\
&\le \int_1^\infty \int_\varepsilon^1 \int_{-v}^v \int_{-u}^u \frac{\|\phi\|_{0,0,1,1}}{u^2v} \,ds\,dt\,dv\,du \le 4\|\phi\|_{0,0,1,1}.
\end{align*}
In light of equation~\eqref{Iuv}, these calculations show that
\begin{equation} \label{binary}
\bigl| \langle PV_2,\phi \rangle \bigr| = \biggl| \lim_{\varepsilon\rightarrow 0} \int_{|u|\ge \varepsilon}\int_{|v|\ge \varepsilon}\frac{\phi(u,v)}{uv}\,dv\,du \biggr| \le 4 \bigl( \|\phi\|_{1,1,0,0} + \|\phi\|_{1,0,0,1} + \|\phi\|_{0,1,1,0} + \|\phi\|_{0,0,1,1} \bigr)
\end{equation}
and that the limit in part~(d) exists. Moreover, a bound of the shape~\eqref{binary} is enough to guarantee that $PV_2$ is continuous (see~\cite[page~134]{ReedSimon}). Since linearity is trivial, we conclude that $PV_2$ is a tempered distribution.

As for the other parts of the proposition, part~(c) can be confirmed by replacing $\phi(u,v)$ with say $\frac u2 e^{-|u|} \phi(0,v)$, in part~(d). Part~(b) can be confirmed similarly, and part~(a) follows in the same way or by direct verification.
\end{proof}

\subsection{Application to the first quadrant density}

Given a modulus $q\ge3$ and a reduced residue class~$a\mod q$, we would like to compute $\delta^{++}_a$, the logarithmic density of the set 
\[
\{x\ge0\colon \Edot^\psi(x;q,a) > 0,\, \Edot^\psi(x;q,1) > 0 \}.
\]
By Proposition~\ref{proposition ANS}, this quantity is the same as
\begin{equation} \label{Heaviside2}
\delta^{++}_a = \mathbb{P}(X_1>0,\,X_2>0) = \int_0^{\infty}\int_0^{\infty} \,d\mu_{X} = \iint_{\mathbb{R}^2}\mathbbm1_{[0,\infty)\times[0,\infty)}(x,y) \,d\mu_{X}
\end{equation}
in the notation of Definition~\ref{X def}. To evaluate the right-hand side, we will use Plancherel's theorem so that we can access our information about the characteristic function~$\phi_X$, which is an element of~$\mathcal{S}(\R^2)$ by~\cite[Lemma 2.2]{BSRPNR}. Consequently, we need to calculate the (inverse) Fourier transform of the tempered distribution corresponding to integrating against the function $\mathbbm1_{[0,\infty)\times[0,\infty)}$.

\begin{prop} \label{proposition fourier inverse indicator function}
In the sense of tempered distributions,
\begin{align*}
\mathcal{F}(\mathbbm1_{[0,\infty)\times[0,\infty)}) &= \pi^2\delta_2 - i\pi(PV_\ell+PV_r)-PV_2 \\
\mathcal{F}^{-1}(\mathbbm1_{[0,\infty)\times[0,\infty)}) &= \frac1{4}\delta_2+\frac{i}{4\pi}(PV_\ell+PV_r)-\frac1{4\pi^2}PV_2.
\end{align*}
\end{prop}

\begin{rem}
The second formula is analogous to the two-dimensional case of~\cite[Theorem~4]{BSRPNR}; both formulas can also be derived abstractly using the fact that the tempered distribution corresponding to integrating against $\mathbbm1_{[0,\infty)\times[0,\infty)}$ is the tensor product of the Heaviside distribution with itself. On the other hand, a direct and concrete proof is short enough to be worth including in its entirety.
\end{rem}

\begin{proof}
For any Schwartz function $\phi\in\mathcal{S}(\R^2)$,
\begin{align*}
\langle\mathbbm1_{[0,\infty)\times[0,\infty)},\mathcal{F}(\phi)\rangle&=\iint \mathbbm1_{[0,\infty)\times[0,\infty)}(\xi,\eta)\iint \phi(x,y)e^{-i(x\xi+y\eta)}  \,dx\,dy \,d\xi\,d\eta \\
&=\lim_{\varepsilon\rightarrow 0}\iint \mathbbm1_{[0,\infty)\times[0,\infty)}(\xi,\eta)\iint \phi(x,y)e^{-i(x\xi+y\eta)}e^{-\varepsilon(\xi+\eta)}  \,dx\,dy \,d\xi\,d\eta \\
&=\lim_{\varepsilon\rightarrow 0}\iint  \phi(x,y)\left(\int_0^\infty e^{-\xi(ix+\varepsilon)}d\xi\right) \left(\int_0^\infty e^{-\eta(iy+\varepsilon)}d\eta\right)  \,dx\,dy \\
&=\lim_{\varepsilon\rightarrow 0}\iint  \phi(x,y)\frac1{(ix+\varepsilon)(iy+\varepsilon)}  \,dx\,dy \\
&=\lim_{\varepsilon\rightarrow 0}\iint  \phi(x,y)\frac1{(ix+\varepsilon)(iy+\varepsilon)}  \,dx\,dy \\
&=\lim_{\varepsilon\rightarrow 0}\iint  \phi(x,y)\frac{\varepsilon^2-i\varepsilon (x+y)-xy}{(x^2+\varepsilon^2)(y^2+\varepsilon^2)}  \,dx\,dy.
\end{align*}
We evaluate this last integral in four parts. In the first part, we use the change of variables $u=x/\varepsilon$ and $v=y/\varepsilon$ to see that
\begin{align*}
\lim_{\varepsilon\rightarrow 0}\varepsilon^2\iint \frac{\phi(x,y)}{(x^2+\varepsilon^2)(y^2+\varepsilon^2)} \,dx\,dy&= \lim_{\varepsilon\rightarrow 0}\iint  \frac{\phi(\varepsilon u,\varepsilon v)}{(1+u^2)(1+v^2)} \,du\,dv = \pi^2\phi(0,0)=\langle\pi^2\delta_2,\phi\rangle.
\end{align*}
In the second one, we set $v=y/\varepsilon$ to obtain
\begin{align*}
-i\lim_{\varepsilon\rightarrow 0} \varepsilon \iint \frac{\phi(x,y) x}{(x^2+\varepsilon^2)(y^2+\varepsilon^2)} \,dx\,dy=-i\lim_{\varepsilon\rightarrow 0} \iint \frac{\phi(x,\varepsilon v) x}{(x^2+\varepsilon^2)(1+v^2)} \,dx\,dv&=\langle-i\pi PV_\ell,\phi\rangle
\end{align*}
in the notation of Proposition~\ref{singular tempered};
an analogous argument shows that
\begin{align*}
-i\lim_{\varepsilon\rightarrow 0} \varepsilon \iint \frac{\phi(x,y)y}{(x^2+\varepsilon^2)(y^2+\varepsilon^2)} \,dx\,dy=\langle-i\pi PV_r,\phi\rangle.
\end{align*}
Finally,
\begin{align*}
-\lim_{\varepsilon\rightarrow 0}\iint \frac{\phi(x,y)xy}{(x^2+\varepsilon^2)(y^2+\varepsilon^2)} \,dx\,dy&=\langle-PV_2,\phi\rangle.
\end{align*}
These establish the first assertion of the proposition.  The second assertion follows directly from equation~\eqref{F inverse}.
\end{proof}

We are now ready to obtain an expression for the logarithmic density of interest in terms of the characteristic function~$\phi_X$ (analogous to formulas in~\cite[end of Section~2D]{BSRPNR}).

\begin{prop} \label{density formula prop}
Assume GRH and LI. For any modulus $q\ge3$ and any reduced residue class $a\mod q$,
\[
\delta^{++}_a = \frac1{4}-\frac1{4\pi^2} \lim_{\varepsilon\rightarrow 0}\int_{|\xi|>\varepsilon} \int_{|\eta|>\varepsilon}\frac{\phi_X(\xi,\eta)-\phi_X(0,\eta)\phi_X(\xi,0)}{\xi\eta}\,d\xi\,d\eta.
\]
\end{prop}

\begin{proof}
By equation~\eqref{Heaviside2} and Proposition~\ref{proposition fourier inverse indicator function},
\begin{align*}
\delta^{++}_a &= \mathbb{P}(X\in \R_{>0}\times\R_{>0}) = \langle\mathbbm1_{[0,\infty)\times[0,\infty)},d\mu\rangle \\
&= \bigl\langle \mathbbm1_{[0,\infty)\times[0,\infty)}, \mathcal{F}^{-1}\bigl( \mathcal{F}(d\mu) \bigr) \bigr\rangle = \bigl\langle \mathcal{F}^{-1}(\mathbbm1_{[0,\infty)\times[0,\infty)}),\phi_X \bigr\rangle \\
&=\frac1{4}\phi_X(0,0)+\frac{i}{4\pi}\langle PV_\ell+PV_r,\phi_X\rangle-\frac{i}{4\pi^2}\langle PV_2,\phi_X\rangle.
\end{align*}
For the first term, $\phi_X(0,0)=1$ since~$d\mu$ is a probability measure. Next,
\[
\langle PV_\ell,\phi_X\rangle=\lim_{\varepsilon\rightarrow 0}\int_{|\xi|>\varepsilon}\frac{\phi_X(\xi,0)}{\xi}d\xi=\lim_{\varepsilon\rightarrow 0}\int_{\varepsilon}^\infty \frac{\phi_X(\xi,0)-\phi_X(-\xi,0)}{\xi}d\xi = 0
\]
since $\phi_X(\xi,0)$ is an even function; similarly $\langle PV_r,\phi_X\rangle = 0$, as well. Finally
\[
\langle PV_2,\phi_X\rangle=\lim_{\varepsilon\rightarrow 0}\int_{|\xi|>\varepsilon}\int_{|\eta|>\varepsilon}\frac{\phi_X(\xi,\eta)}{\xi\eta}\,d\xi\,d\eta=\lim_{\varepsilon\rightarrow 0}\int_{|\xi|>\varepsilon}\int_{|\eta|>\varepsilon}\frac{\phi_X(\xi,\eta)-\phi_X(0,\eta)\phi_X(\xi,0)}{\xi\eta}\,d\xi\,d\eta,
\]
again since $\phi_X(0,\eta)\phi_X(\xi,0)$ is an even function (of either variable).
\end{proof}

\section{Evaluating double integrals with rigorous error bounds} \label{general numeric section}

We have seen in Proposition~\ref{density formula prop} that the key to understanding $\delta_a^{++}$ is the limiting integral
\begin{equation} \label{equation definition bias}
I=\lim_{\upsilon\rightarrow 0}\int_{|\xi|>\upsilon}\int_{|\eta|>\upsilon}\frac{\phi_X(\xi,\eta)-\phi_X(0,\eta)\phi_X(\xi,0)}{\xi\eta}\,d\xi\,d\eta
\end{equation}
(where we have changed~$\varepsilon$ to~$\upsilon$ to avoid an imminent clash of notation). In the method of Feuerverger and the second author~\cite{BSRPNR}, which itself was adapted from Rubinstein and Sarnak~\cite{RS94}, three steps are required to convert this unbounded integral of a function involving an infinite product into a finitely calculable quantity.
We devote the next sections to each of these steps---discretizing the integral, bounding the range of summation of the resulting sum, and truncating the infinite product---with attention to rigorous bounds for the error terms in each step. All of the material in this section applies to general moduli~$q$; we specialize to $q=11$ in Section~\ref{section numerics}.

\subsection{Discretizing the integral}

To discretize the integral~$I$, we introduce the quantities
\begin{equation} \label{E1 def}
S(\varepsilon) = 4 \sum_{m,n\ \text{odd}} \frac{\phi_X(\frac{m\varepsilon}{2},\frac{n\varepsilon}{2})}{mn}
\quad\text{and}\quad
E_1(\varepsilon) = I - S(\varepsilon).
\end{equation}
(All indices of summation are restricted to integer values throughout.)
We bound this first error~$E_1(\varepsilon)$ using the Poisson summation formula and the symmetries of our integrand. Fortunately, we are able to invoke prior work to quickly dispose of this step.

\begin{prop}\label{proposition summary FM E1}
Let $\Xva=(\Xva_1,\Xva_2)$ be the random variable from Definition~\ref{X def}.
For any $\varepsilon>0$.
\[
|E_1(\varepsilon)| \le 8\pi^2 \mathop{\sum\sum}_{\substack{\kappa,\lambda\ge0\\ 
(\kappa,\lambda)\neq(0,0)}} \min\bigl\{ \mathbb{P}(X_1>2\pi\kappa/\varepsilon), \mathbb{P}(X_2>2\pi\lambda/\varepsilon) \bigr\}.
\]
\end{prop}

\begin{proof} \label{use symmetry}
It was shown in~\cite[equations~(3-7) and~(3-9)]{BSRPNR} that
\begin{multline}
\biggl| I - 4 \sum_{\substack{m,n\ \text{odd}}}\frac{\phi_X(\frac{m\varepsilon}{2},\frac{n\varepsilon}{2}) - \phi_X(\frac{m\varepsilon}{2},0) \phi_X(0,\frac{n\varepsilon}{2})}{mn} \biggr| \\
\le 8\pi^2 \mathop{\sum\sum}_{\substack{\kappa,\lambda\ge0\\ 
(\kappa,\lambda)\neq(0,0)}} \min\bigl\{ \mathbb{P}(X_1>2\pi\kappa/\varepsilon), \mathbb{P}(X_2>2\pi\lambda/\varepsilon) \bigr\}.
\end{multline}
(Unfortunately, a factor of~$2$ was omitted from the right-hand side of~\cite[equation~(3-8)]{BSRPNR}.) That calculation was carried out for an integrand involving a specific function $\hat\rho_{8;3,5,7}$, but the calculation is valid for any Schwartz function symmetric about the origin, including~$\phi_X$. The proposition follows upon noting that the part of the left-hand sum involving ${\phi_X(\frac{m\varepsilon}{2},0) \phi_X(0,\frac{n\varepsilon}{2})}/{mn}$ vanishes due to the fact that~$\phi_X$ is even in either variable.
\end{proof}

\subsection{Truncating the range of summation}

To truncate the infinite sum $S(\varepsilon)$, we now introduce the quantities
\begin{equation} \label{SepsilonC prelim and E2}
S(\varepsilon,C) = \sum_{\substack{|m|,|n|\le C \\ m,n\ \text{odd}}}\frac{\phi_X(\frac{m\varepsilon}{2},\frac{n\varepsilon}{2})}{mn}
\quad\text{and}\quad
E_2(\varepsilon,C) = \frac1{4}S(\varepsilon) - S(\varepsilon,C)
\end{equation}
for any real $C\ge1$, so that
\begin{equation} \label{sneaky symmetry}
\bigl| E_2(\varepsilon,C) \bigr| \le \sum_{\substack{|m|>C \text{ or } |n|>C \\ m,n\ \text{odd}}} \biggl| \frac{\phi_X(\frac{m\varepsilon}{2},\frac{n\varepsilon}{2})}{mn} \biggr| \le 4 \sum_{\substack{m>C \\ |n| \le m \\ m,n\ \text{odd}}} \biggl| \frac{\phi_X(\frac{m\varepsilon}{2},\frac{n\varepsilon}{2})}{mn} \biggr|
\end{equation}
thanks to the symmetries noted at the end of Definition~\ref{X def} (we have double-counted those points for which $|m|=|n|$, but the upper bound remains valid).
In the arguments to come we will need to treat individual Dirichlet characters differently, and so we introduce the following notation.

\begin{definition}
Fix a modulus $q\ge3$. For every Dirichlet character $\chi\mod q$, define
\begin{align} \label{eq.def of Fzchi}
F(z,\chi) = \prod_{\substack{\gamma>0\\ L(1/2+i\gamma,\chi)=0}} J_0\left( \frac{2z}{\sqrt{1/4+\gamma^2}} \right),
\end{align}
which allows us to rewrite equation~\eqref{characteristicfunctionX} as
\begin{align} \label{eq:phiX with F}
\phi_X(t_1,t_2) = \prod_{\substack{\chi\mod{q}\\ \chi\neq\chi_0}} F(\left|t_1+\chi(a)t_2\right|,\chi).
\end{align}
Further define
\[
B_2(a,\varepsilon,C) = \sum_{\substack{m>C \\ |n| \le m \\ m,n\ \text{odd}}} \biggl| \frac1{mn} \prod_{\substack{\chi \mod{q}\\ \chi\neq \chi_0}} F\biggl( \frac{\varepsilon}{2} |m+\chi(a)n|,\chi \biggr) \biggr|,
\]
so that the upper bound in equation~\eqref{sneaky symmetry} is simply $\bigl| E_2(\varepsilon,C) \bigr| \le 4B_2(a,\varepsilon,C)$.
\end{definition}

It is harder to get an upper bound for $F\bigl( \frac{\varepsilon}{2} |m+\chi(a)n|,\chi \bigr)$ when $\Re\chi(a)$ is close to $-\frac mn$, since $F(t,\chi)$ is largest when $t\in\R$ is near~$0$. Consequently, for various sizes of~$m$ and~$n$ we choose to restrict to appropriate subsets of characters when creating such upper bounds.
The following proposition, which implements this strategy, is written in as flexible a form as possible to facilitate use by future researchers; it is not as ugly as it seems.

\begin{prop}\label{proposition bound 2}
Fix a modulus $q\ge3$. For each nontrivial character $\chi\mod q$, fix positive constants $d_+(\chi),d_-(\chi),d(\chi)$ and $e_+(\chi),e_-(\chi),e(\chi)$ such that
\begin{align} \label{F bound polynomial}
|F(x,\chi)|\le \min\{ 1, d_+(\chi)|x|^{-e_+(\chi)}, d_-(\chi)|x|^{-e_-(\chi)}, d(\chi)|x|^{-e(\chi)} \}
\end{align}
for all $x\in\R$.
Choose real numbers $b>1$, $c\in[0,1]$, $c_+\in[0,1]$ and $c_-\in[-1,0]$, and set
\[
\begin{alignedat}{2}
d_{c_+} &= \prod_{\substack{\chi \mod{q}\\ \Re(\chi(a))\ge c_+}} d_+(\chi) &\quad\text{and}\quad& e_{c_+} = \sum_{\substack{\chi \mod{q}\\ \Re(\chi(a))\ge c_+}}e_+(\chi) \\
d_{c_-} &= \prod_{\substack{\chi \mod{q}\\ \Re(\chi(a))\le c_-}} d_-(\chi) &\quad\text{and}\quad& e_{c_-} = \sum_{\substack{\chi \mod{q}\\ \Re(\chi(a))\le c_-}} e_-(\chi) \\
d_c &= \prod_{\substack{\chi \mod{q}\\ |\Re(\chi(a))|\le c}} d(\chi) &\quad\text{and}\quad& e_c = \sum_{\substack{\chi \mod{q}\\ |\Re(\chi(a))|\le c}} e(\chi).
\end{alignedat}
\]
If $e_c>1$, then for any real numbers $\varepsilon>0$ and $C\ge1$,
\begin{equation} \label{triple threat}
B_2(a,\varepsilon,C)\le B_2^+(a,\varepsilon,C)+B_2^-(a,\varepsilon,C)+\widetilde{B}_2(a,\varepsilon,C),
\end{equation}
where
\begin{align*}
B_2^+(a,\varepsilon,C) &= \frac{bd_{c_+}}{4e_{c_+}}\left(1-\frac1{b}\right)\left(\varepsilon\left(1 +\frac{c_+}{b}\right)\right)^{-e_{c_+}}\left(\left\lfloor\frac{C}{2}\right\rfloor-1\right)^{-e_{c_+}} \\
B_2^-(a,\varepsilon,C) &= \frac{bd_{c_-}}{4e_{c_-}}\left(1-\frac1{b}\right)\left(\varepsilon\left(1 -\frac{c_-}{b}\right)\right)^{-e_{c_-}}\left(\left\lfloor\frac{C}{2}\right\rfloor-1\right)^{-e_{c_-}} \\
\widetilde{B}_2(a,\varepsilon,C) &= \frac{d_c}{e_c-1}\left(\frac C{2b}+1 \right) \left(\varepsilon\left(1-\frac{c}{b}\right)\right)^{-e_c}\left(\left\lfloor\frac{C}{2}\right\rfloor-1\right)^{-e_c}.
\end{align*}
\end{prop}

\begin{rem}
The existence of suitable constants $d(\chi)$ and $e(\chi)$ in equation~\eqref{F bound polynomial} is assured by~\cite[equation~(3.19)]{BSRPNR}; there is a trade-off between the sizes of $d(\chi)$ and $e(\chi)$, and the three variants present allow for different choices to be made in different parts of the estimation argument. In Section~\ref{section numerics} we demonstrate how to compute such constants in the case $q=11$.
\end{rem}

\begin{proof}
For each~$m$ in the sum defining $B_2(a,\varepsilon,C)$, we break the sum over~$n$ into three ranges depending on its size relative to~$m$. In each range, we are free to omit whichever characters~$\chi$ we wish from the product, thanks to the bound $|F(x,\chi)| \le 1$ from the hypothesis~\eqref{F bound polynomial}. Specifically, when~$n$ is large and positive we keep only characters~$\chi$ for which $\Re\chi(a)$ is sufficiently positive; when~$n$ is large and negative we keep only those for which $\Re\chi(a)$ is sufficiently negative; and when~$|n|$ is small we keep only those for which $|\Re\chi(a)|$ is also small.
In this manner we write
\begin{align}
B_2(a,\varepsilon,C) &\le \sum_{\substack{m>C \\ m\text{ odd}}} \Biggl( \sum_{\substack{\frac{m}{b}<n\le m\\ n\text{ odd}}}\frac1{|mn|} \prod_{\substack{\chi \mod{q}\\ \chi\neq \chi_0\\ \Re(\chi(a))\ge c_+}} \biggl| F\biggl( \frac{\varepsilon}{2} |m+\chi(a)n|,\chi \biggr) \biggr| \notag \\
&\qquad{}+\sum_{\substack{-m \le n < -\frac{m}{b} \\ n\text{ odd}}}\frac1{|mn|} \prod_{\substack{\chi \mod{q}\\ \chi\neq \chi_0\\ \Re(\chi(a))\le c_-}} \biggl| F\biggl( \frac{\varepsilon}{2} |m+\chi(a)n|,\chi \biggr) \biggr| \notag \\
&\qquad{}+ \sum_{{\substack{|n| \le \frac{m}{b}\\ n\text{ odd}}}}\frac1{|mn|} \prod_{\substack{\chi \mod{q}\\ \chi\neq \chi_0 \\ |\Re(\chi(a))|\le c}} \biggl| F\biggl( \frac{\varepsilon}{2} |m+\chi(a)n|,\chi \biggr) \biggr| \Biggr) \notag \\
&\le \sum_{\substack{m>C \\ m\text{ odd}}} \Biggl( b \sum_{\substack{\frac{m}{b}<n\le m\\ n\text{ odd}}}\frac1{m^2} \prod_{\substack{\chi \mod{q}\\ \chi\neq \chi_0\\ \Re(\chi(a))\ge c_+}} d_+(\chi)\left|\frac{\varepsilon}{2}(m+\chi(a) n)\right|^{-e_+(\chi)} \label{4of6} \\
&\qquad{}+ b \sum_{\substack{-m \le n < -\frac{m}{b} \\ n\text{ odd}}}\frac1{m^2} \prod_{\substack{\chi \mod{q}\\ \chi\neq \chi_0\\ \Re(\chi(a))\le c_-}} d_-(\chi)\left|\frac{\varepsilon}{2}(m+\chi(a) n)\right|^{-e_-(\chi)} \notag \\
&\qquad{}+ \sum_{{\substack{|n| \le \frac{m}{b}\\ n\text{ odd}}}} \frac1m \prod_{\substack{\chi \mod{q}\\ \chi\neq \chi_0 \\ |\Re(\chi(a))|\le c}} d(\chi)\left|\frac{\varepsilon}{2}(m+\chi(a) n)\right|^{-e(\chi)} \Biggr). \notag
\end{align}
In the first product we have $|m+\chi(a)n| \ge \Re(m+\chi(a)n) \ge m+c_+n > m(1+\frac{c_+}b)$, while in the second product we have $|m+\chi(a)n| \ge \Re(m-\chi(a)n) \ge m-c_-n \ge m(1-\frac{c_-}b)$. Thus the first term on the right-hand side is bounded by
\begin{align*}
\sum_{\substack{m>C \\ m\text{ odd}}} b \sum_{\substack{\frac{m}{b}<n\le m\\ n\text{ odd}}}\frac1{m^2} & \prod_{\substack{\chi \mod{q}\\ \chi\neq \chi_0\\ \Re(\chi(a))\ge c_+}} d_+(\chi) \Bigl( \frac{\varepsilon}{2} m \bigl( 1+\frac{c_+}b \bigr) \Bigr)^{-e_+(\chi)} \\
&= b d_{c_+} \Bigl( \frac{\varepsilon}{2} \bigl( 1+\frac{c_+}b \bigr) \Bigr)^{-e_{c_+}} \sum_{\substack{m>C \\ m\text{ odd}}} m^{-e_{c_+}-2} \sum_{\substack{\frac{m}{b}<n\le m\\ n\text{ odd}}} 1 \\
&\le b d_{c_+} \Bigl( \frac{\varepsilon}{2} \bigl( 1+\frac{c_+}b \bigr) \Bigr)^{-e_{c_+}} \frac12\Bigl( 1-\frac1b \Bigr) \sum_{\substack{m>C \\ m\text{ odd}}} m^{-e_{c_+}-1}.
\end{align*}
Since
\[
\sum_{\substack{m>C \\ m\text{ odd}}} m^{-e_{c_+}-1} \le \sum_{\substack{k\ge\lfloor C/2\rfloor}} (2k+1)^{-e_{c_+}-1} < \sum_{\substack{k\ge\lfloor C/2\rfloor}} (2k)^{-e_{c_+}-1} < 2^{-e_{c_+}-1} \frac1{e_{c_+}} \left(\left\lfloor\frac{C}{2}\right\rfloor-1\right)^{-e_{c_+}}
\]
by comparison to an integral, the first term on the right-hand side of equation~\eqref{4of6} is bounded above by $B_2^+(a,\varepsilon,C)$. A parallel argument shows that the second term on that right-hand side is bounded above by $B_2^-(a,\varepsilon,C)$.

Finally, in the last product in equation~\eqref{4of6}, we have $|m+\chi(a)n| \ge |\Re(m+\chi(a)n)| \ge m-n\Re\chi(a) \ge m(1-\frac{c}b)$, and thus
\begin{align*}
\sum_{\substack{m>C \\ m\text{ odd}}} \sum_{{\substack{|n| \le \frac{m}{b}\\ n\text{ odd}}}} \frac1m & \prod_{\substack{\chi \mod{q}\\ \chi\neq \chi_0 \\ |\Re(\chi(a))|\le c}} d(\chi) \left( \frac{\varepsilon}{2} m \left( 1-\frac{c}b \right) \right)^{-e(\chi)} \\
&= d_c\left(\frac{\varepsilon}{2}\left(1-\frac{c}{b}\right)\right)^{-e_c} \sum_{\substack{m>C \\ m\text{ odd}}} m^{-e_c-1} \sum_{{\substack{|n| \le \frac{m}{b}\\ n\text{ odd}}}} 1 \\
&= d_c\left(\frac{\varepsilon}{2}\left(1-\frac{c}{b}\right)\right)^{-e_c} \sum_{\substack{m>C \\ m\text{ odd}}} m^{-e_c-1} \Bigl( \frac mb+2 \Bigr).
\end{align*}
Again by comparisons to integrals,
\begin{align*}
\sum_{\substack{m>C \\ m\text{ odd}}} m^{-e_c-1} \frac mb &< \frac1b 2^{-e_c} \frac1{e_c-1} \left(\left\lfloor\frac{C}{2}\right\rfloor-1\right)^{-e_c+1}< \frac C{2b} 2^{-e_c} \frac1{e_c-1} \left(\left\lfloor\frac{C}{2}\right\rfloor-1\right)^{-e_c} \\
\sum_{\substack{m>C \\ m\text{ odd}}} 2m^{-e_c-1} &< 2 \cdot 2^{-e_c-1} \frac1{e_c} \left(\left\lfloor\frac{C}{2}\right\rfloor-1\right)^{-e_c}< 2^{-e_c} \frac1{e_c-1} \left(\left\lfloor\frac{C}{2}\right\rfloor-1\right)^{-e_c},
\end{align*}
and thus the last term on the right-hand side of equation~\eqref{4of6} is bounded above by $\widetilde{B}_2(a,\varepsilon,C)$, which completes the proof.
\end{proof}

\subsection{Truncating the product over the zeros}

Our approximation $S(\varepsilon,C)$ still involves the infinite products $F(z,\chi)$ defined in equation~\eqref{eq.def of Fzchi}, which we must truncate to finite products in order to carry out numerical calculations. A bound for the error involved in this last approximation step is fairly clean to state, but only once we introduce several more pieces of helpful notation. We always have in mind some fixed reduced residue class $a\mod q$ that these quantities depend upon; we will suppress~$a$ from the notation in this section for readability.

\begin{definition} \label{alpha def}
For any Dirichlet charater $\chi\mod q$, write
\[
F(z,\chi) = \prod_{\substack{\gamma>0 \\ L(1/2+i\gamma,\chi)=0}} J_0(\alpha_\gamma z)
\quad\text{with}\quad
\alpha_\gamma=\frac{2}{\sqrt{1/4+\gamma^2}}.
\]
Furthermore, when~$\chi$ is a nonreal complex character, define
\[
\widetilde{F}(z,\chi) = F(z,\chi)F(z,\overline{\chi}) = \prod_{\substack{\gamma\in\R \\ L(1/2+i\gamma,\chi)=0}} J_0(\alpha_\gamma z)
\]
by the symmetries of zeros of Dirichlet $L$-functions .
\end{definition}

Our next piece of notation is motivated by the idea of replacing the tail of an infinite product with its quadratic (and indeed cubic) Maclaurin approximation.

\begin{definition} \label{Delta_T def}
For any Dirichlet charater $\chi\mod q$ and any $T\ge0$, define
\begin{align*}
b_1(T,\chi)=-\sum_{\substack{{\gamma\ge T} \\ L(1/2+i\gamma,\chi)=0}} \frac1{1/4+\gamma^2}
\quad\text{and}\quad
\widetilde{b}_1(T,\chi) = -\sum_{\substack{{|\gamma|\ge T} \\ L(1/2+i\gamma,\chi)=0}} \frac1{1/4+\gamma^2},
\end{align*}
noting that if~$\chi$ is real and $L(\frac12,\chi)\ne0$ then $\widetilde{b}_1(0,\chi) = 2{b}_1(0,\chi)$.
Further define
\begin{align*}
F_T(z,\chi) &= \biggl( \prod_{\substack{{0<\gamma<T} \\ L(1/2+i\gamma,\chi)=0}} J_0(\alpha_\gamma z) \biggr) (1+b_1(T,\chi)z^2) \\
\widetilde{F}_T(z,\chi) &= \biggl(\prod_{\substack{{-T<\gamma<T} \\ L(1/2+i\gamma,\chi)=0}} J_0(\alpha_\gamma z) \biggr) (1+\widetilde{b}_1(T,\chi)z^2).
\end{align*}
We now set
\begin{align*}
\Delta_T(z,\chi) &= \biggl( \frac1{1+b_1(T,\chi)z^2} {\prod_{\substack{{\gamma\ge T} \\ L(1/2+i\gamma,\chi)=0}} J_0(\alpha_\gamma z)} \biggr) -1 \notag \\
\widetilde{\Delta}_T(z,\chi) &= \biggl( \frac1{1+\widetilde{b}_1(T,\chi)z^2}
{\prod_{\substack{{|\gamma|\ge T} \\ L(1/2+i\gamma,\chi)=0}} J_0(\alpha_\gamma z)} \biggr) -1,
\end{align*}
so that $F(z,\chi)=F_T(z,\chi)(1+\Delta_T(z,\chi))$ and $\widetilde{F}(z,\chi)=\widetilde{F}_T(z,\chi)(1+\widetilde{\Delta}_T(z,\chi))$.
\end{definition}

\begin{lemma} \label{Delta_T lemma}
For any Dirichlet charater $\chi\mod q$ and any $T>0$ and $x\in\R$,
\begin{align*}
|\Delta_T(x,\chi)| &< \frac{b_1(T,\chi)^2x^4}{2(1-|b_1(T,\chi)|x^2)^2},\quad\text{if } |b_1(T,\chi)|x^2<1 \\
\bigl| \widetilde{\Delta}_T(x,\chi) \bigr| &< \frac{\widetilde{b}_1(T,\chi)^2x^4}{2(1-|\widetilde{b}_1(T,\chi)|x^2)^2},\quad \text{if } |\widetilde{b}_1(T,\chi)|x^2<1.
\end{align*}
\end{lemma}

\begin{proof}
The first inequality is precisely~\cite[equation~(3-27)]{BSRPNR} translated into our notation; the proof, as given in~\cite[Section~4.3]{RS94}, proceeds by noticing that all coefficients of the Maclaurin expansion of $J_0(\alpha_\gamma z)$ are smaller than those of $e^{\alpha_\gamma^2z^2/4}$. The inequalities in those works were presented for specific $L$-functions, but the proof works for any sequence of values of~$\gamma$ that grow quickly enough to make the product converge. The second inequality is an immediate consequence of this observation.
\end{proof}

Given equations~\eqref{SepsilonC prelim and E2} and~\eqref{eq:phiX with F}, we may write
\[
S(\varepsilon, C)=\sum_{\substack{|m|,|n|\le C\\ m,n\ \text{odd}}}\frac1{mn}{\prod_{\substack{\chi \mod{q}\\ \chi\neq \chi_0}} F\left(\frac{\varepsilon}{2}|m+\chi(a)n|,\chi\right)}.
\]
We wish to pair together the factors corresponding to nonreal complex characters~$\chi$; let $H(q)$ be a set of Dirichlet characters$\mod q$ consisting of exactly one character from each pair of complex-conjugate characters (and no real characters), and let $R(q)$ denote the set of nonprincipal real characters$\mod q$, so that
\[
S(\varepsilon, C)=\sum_{\substack{|m|,|n|\le C\\ m,n\ \text{odd}}}\frac1{mn} \prod_{q\in R(q)} F\left(\frac{\varepsilon}{2}|m+\chi(a)n|,\chi\right) \prod_{\chi\in H(q)} \widetilde{F}\left(\frac{\varepsilon}{2}|m+\chi(a)n|,\chi\right).
\]
We form an approximation using finite products by defining
\begin{equation} \label{S(varepsilon,C,T)}
S(\varepsilon,C,T)=\sum_{\substack{|m|,|n|\le C\\ m,n\ \text{odd}}}\frac1{mn}{\prod_{\chi\in R(q)} F_T\left(\frac{\varepsilon}{2}|m+\chi(a)n|,\chi\right)\prod_{\chi\in H(q)} \widetilde{F}_T\left(\frac{\varepsilon}{2}|m+\chi(a)n|,\chi\right)}
\end{equation}
and the corresponding error
\begin{equation} \label{E3 def}
E_3(\varepsilon,C,T) = S(\varepsilon, C) - S(\varepsilon,C,T).
\end{equation}

\begin{definition}
For any modulus $q\ge3$ and any real numbers $T>0$ and~$x$, define
\[
\widehat{b}_1(T)=\max\biggl\{ \max\limits_{\chi\in R(q)} |b_1(T,\chi)|, \max\limits_{\chi\in H(q)} |\widetilde{b}_1(T,\chi)| \biggr\}
\quad\text{and}\quad
D(x,T)=\frac{\widehat{b}_1(T)^2x^4}{2(1-|\widehat{b}_1(T)|x^2)^2}.
\]
\end{definition}

\begin{prop} \label{E_3(varepsilon,C,T)}
Fix a modulus $q\ge3$. Let $\varepsilon>0$ and $C\ge1$ and $T>0$ be real numbers.
If $\widehat{b}_1(T) \varepsilon^2C^2<1$, then
\begin{align*}
|E_3(\varepsilon,C,T)| &\le \sum_{\substack{|m|,|n|\le C\\ m,n\ \text{odd}}} \Biggl( \biggl| \frac1{mn}{\prod_{\chi\in R(q)} F_T\left(\frac{\varepsilon}{2}|m+\chi(a)n|,\chi\right)\prod_{\chi\in H(q)} \widetilde{F}_T\left(\frac{\varepsilon}{2}|m+\chi(a)n|,\chi\right)} \biggr| \\
&\qquad{}\times \biggl(\prod_{\chi\in R(q)} \Bigl( 1+D\Bigl( \frac{\varepsilon}{2}|m+\chi(a)n|,T \Bigr) \Bigr) \prod_{\chi\in H(q)} \Bigl( 1+D\Bigl( \frac{\varepsilon}{2}|m+\chi(a)n|,T \Bigr) \Bigr) -1 \biggr) \Biggr).
\end{align*}
\end{prop}

\begin{proof}
The hypothesized inequality forces both $|b_1(T,\chi)|(\frac\varepsilon2|m+\chi(a)n|)^2<1$ for every $\chi\in R(q)$ and $|\widetilde{b}_1(T,\chi)|(\frac\varepsilon2|m+\chi(a)n|)^2<1$ for every $\chi\in H(q)$. The proposition then follows directly from Definition~\ref{Delta_T def} and Lemma~\ref{Delta_T lemma}.
\end{proof}

\section{Numerical computations modulo \texorpdfstring{$11$}{11}}\label{section numerics}

In this section, we specify the results of the previous section to the case $q=11$, making specific choices of the parameters in those results and obtaining numerical values for $\delta_a^{++}$ with rigorous error bounds.

In order to proceed with computations modulo~$11$, we first label the Dirichlet characters modulo~$11$. We identify Dirichlet characters by the values they take at $2\mod{11}$, which is a primitive root modulo~$11$.

\begin{definition} \label{labeling characters}
For $0\le j\le 9$ we write $\chi_j$ for the Dirichlet character$\mod{11}$ such that $\chi_j(2) = (e^{2\pi i/10})^j$ (so that $\chi_0$ is the principal character as usual). Note that in this labeling, $\bar\chi_j = \chi_{10-j}$ for $1\le j\le 9$, and in particular $\chi_5$ is the quadratic character (the Legendre symbol) modulo~$11$. This labeling is the same as the labeling in the SageMath library.
\end{definition}

We will need numerical approximations for sums of the form $\sum_\gamma \frac1{1/4+\gamma^2}$, which can be written as either $-b_1(0,\chi)$ or $-\widetilde b_1(0,\chi)$ in the notation of Definition~\ref{Delta_T def}.

\begin{lemma} \label{numerical gamma sums}
Assuming GRH,
\begin{align*}
-\widetilde b_1(0,\chi_1) &\approx 0.371958756757, &
-\widetilde b_1(0,\chi_2) &\approx 0.304226855907, \\
-\widetilde b_1(0,\chi_3) &\approx 0.817510797013, &
-\widetilde b_1(0,\chi_4) &\approx 0.359942299951,
\end{align*}
and $-b_1(0,\chi_5) \approx 0.253756556727$.
\end{lemma}

\begin{proof}
A closed formula for these sums can be obtained from two formulas from~\cite[Corollary~10.18]{MV} related to a constant $B(\chi)$ in the Hadamard product for $L(s,\chi)$. The first is Vorhauer's formula
\[
B(\chi) = -\frac12\log\frac q\pi - \frac{L'}L(1,\bar\chi) + \frac{C_0}2 - \frac{1+\chi(-1)}2\log2
\]
where $C_0$ is the Euler--Mascheroni constant, and the second is
\[
\Re B(\chi) = -\sum_{\substack{\gamma\in\R \\ L(1/2+i\gamma,\chi)=0}} \Re\frac1{1/2+i\gamma} = \frac12 \widetilde b_1(0,\chi)
\]
in our notation; we conclude that
\[
-\widetilde b_1(0,\chi) = \log \frac{q}{\pi}-C_0-(1+\chi(-1))\log 2+2\Re \frac{L'}{L}(1,\chi).
\]
We take $q=11$ (in this formula~$q$ is the conductor of the character~$\chi$ but for a prime modulus every nonprincipal character is primitive). We numerically evaluate the values $\frac{L'}{L}(1,\chi)$ using SageMath, which completes the lemma upon noting that $-b_1(0,\chi_5) = -\frac12\widetilde b_1(0,\chi_5)$.
\end{proof}

These values allow us to provide a bound for the discretization error~$E_1(\varepsilon)$ in Proposition~\ref{proposition summary FM E1}, given the following helpful tail bound of Montgomery~\cite[Theorem~1 of Section~3]{Montgomery80}.

\begin{prop} \label{Montgomery tail}
Let $(\theta_k)$ be a set of independent random variables, each uniformly distributed on $[0,1)$. Let $(r_k)$ be a decreasing sequence with $\lim_{k\rightarrow \infty}r_k=0$ and $\sum_{k=1}^{\infty} r_{k}^{2} < \infty$.
For any integer $K \ge 1$,
\[
\mathbb{P}\biggl(\sum_{k=1}^{\infty} r_{k} \sin (2\pi \theta_{k}) \ge 2 \sum_{k=1}^{K} r_{k}\biggr) \le \exp \biggl(-\frac{3}{4} \biggl(\sum_{k=1}^{K} r_{k}\biggr)^{2} \bigg/ \sum_{k=K+1}^\infty r_{k}^{2} \biggr).
\]
\end{prop}

\begin{lemma}\label{new prob lemma}
Assume GRH.
Let $\Xva=(\Xva_1,\Xva_2)$ be the random variable from Definition~\ref{X def} with $q=11$. For $w\ge2\pi$,
\[
\mathbb{P}(X_1\ge w) \le \exp \bigl(-0.037(w-1.16)^{2}\bigr) \quad\text{and}\quad
\mathbb{P}(X_2\ge w) \le \exp \bigl(-0.037(w-1.16)^{2}\bigr).
\]
\end{lemma}

\begin{proof}
It suffices to consider~$X_1$ as $X_2$ has the same distribution. This random variable is of the shape treated by Proposition~\ref{Montgomery tail}, where $(r_1,r_2,\dots)$ is the decreasing sequence such that, in the notation of Definition~\ref{alpha def},
\[
\{r_k\} = \{\alpha_\gamma\colon \gamma>0, \,L(1/2+i\gamma, \chi)=0\,\,\text{ for some }\chi\mod{11}\},
\]
as multisets. For example, $L(s,\chi_7)$ has a zero $\rho_1 \approx \frac12+i\cdot 1.23119$ (which will play a significant role in Section~\ref{model section}) and this is the smallest positive imaginary part of any zero of a Dirichlet $L$-function modulo~$11$, and hence $r_1 \approx 2/\sqrt{1/4+1.23119^2} \approx 1.50507$. It turns out that $r_2\approx 0.79139$, $r_3\approx 0.72940$, and $r_4\approx 0.57949$. Note that $\frac14 \sum_{k=1}^\infty r_k^2$ equals the sum of the five quantities in Lemma~\ref{numerical gamma sums}, which is approximately $2.107395$. Given the known first three values, it follows that
\[
r_1+r_2+r_3 \approx 3.02586 <3.03 < \pi
\quad\text{and}\quad
\sum_{k=4}^\infty r_{k}^{2} \approx 5.006 \le 5.01.
\]

Now, given~$w$, choose the integer~$K$ such that
\[
\sum_{k=1}^{K} r_{k} \le \frac w2 < \sum_{k=1}^{K+1} r_{k},
\]
noting that $K\ge3$ since $\frac w2\ge\pi$. It follows that
\[
w \ge 2\sum_{k=1}^{K} r_{k} > w - 2r_{K+1} \ge w - 2r_{4}.
\]
By Proposition~\ref{Montgomery tail},
\begin{align*}
\mathbb{P}(X_1 \ge w) \le \mathbb{P}\biggl(X_1 \ge 2 \sum_{k=1}^{K} r_{k}\biggr) &\le \exp \biggl(-\frac{3}{4} \biggl(\sum_{k=1}^{K} r_{k}\biggr)^{2} \bigg/ \sum_{k=K+1}^\infty r_{k}^{2} \biggr) \\
& \le \exp \biggl(-\frac{3}{4}\biggl( \frac{w-2r_4}2 \biggr)^{2} \bigg/ \sum_{k=4}^\infty r_{k}^{2}\biggr),
\end{align*}
which implies the statement of the lemma.
\end{proof}

\begin{rem}
Increasing the lower bound~$2\pi$ in the statement of Lemma~\ref{new prob lemma} would result in a stronger upper bound. However, that restricted range for~$w$ would in turn restrict the allowable values of~$\varepsilon$ in the results of Section~\ref{general numeric section}. The choice of~$2\pi$ allows for the simple range of possibilities $0<\varepsilon<1$.
\end{rem}

\begin{prop} \label{bound E1}
Assume GRH. For $q=11$ and any $2\le a\le 10$, we have $|E_1(0.2)| \le 3.09\times10^{-13}$.
\end{prop}

\begin{proof}
For any $0<\varepsilon<1$, combining Proposition~\ref{use symmetry} with Lemma~\ref{new prob lemma} yields
\begin{align*}
|E_1(\varepsilon)|\le 8\pi^2\sum_{\kappa=0}^\infty \sum_{\lambda=\max\{\kappa,1\}}^\infty \exp\biggl( -0.037\Bigl( \frac{2\pi\lambda}{\varepsilon}-1.16 \Bigr)^2 \biggr).
\end{align*}
Since for any integer $\lambda\ge1$,
\begin{align*}
\exp\biggl( -0.037\Bigl( \frac{2\pi(\lambda+1)}{\varepsilon}-1.16 \Bigr)^2 \biggr) &\le \exp\biggl( -0.037\Bigl( \frac{2\pi\lambda}{\varepsilon}-1.16 \Bigr)^2 \biggr) \exp\left(-0.037\frac{4\pi^2}{\varepsilon^2}\right) \\
&\le \exp\biggl( -0.037\Bigl( \frac{2\pi\lambda}{\varepsilon}-1.16 \Bigr)^2 \biggr) e^{-14.6/\varepsilon^2},
\end{align*}
the inner sum can be bounded by a geometric series with common ratio $e^{-14.6/\varepsilon^2}$, yielding
\begin{align*}
|E_1(\varepsilon)|&\le 8\pi^2 \frac1{1-e^{-14.6/\varepsilon^2}} \exp\biggl( -0.037\Bigl( \frac{2\pi}{\varepsilon}-1.16 \Bigr)^2 \biggr) \\
&\qquad{}+ 8\pi^2 \sum_{\kappa=1}^\infty \frac1{1-e^{-14.6/\varepsilon^2}} \exp\biggl( -0.037\Bigl( \frac{2\pi\kappa}{\varepsilon}-1.16 \Bigr)^2 \biggr) \\
&= \frac{8\pi^2}{1-e^{-14.6/\varepsilon^2}} \exp\biggl( -0.037\Bigl( \frac{2\pi}{\varepsilon}-1.16 \Bigr)^2 \biggr) + \frac{8\pi^2}{(1-e^{-14.6/\varepsilon^2})^2} \exp\biggl( -0.037\Bigl( \frac{2\pi}{\varepsilon}-1.16 \Bigr)^2 \biggr) \\
&= \frac{8\pi^2(2-e^{-14.6/\varepsilon^2})}{(1-e^{-14.6/\varepsilon^2})^2} \exp\biggl( -0.037\Bigl( \frac{2\pi}{\varepsilon}-1.16 \Bigr)^2 \biggr).
\end{align*}
Setting $\varepsilon=0.2$ confirms the statement of the proposition.
\end{proof}

Our next task is to give a bound for the sum truncation error~$E_2(\varepsilon,C)$ in Proposition~\ref{proposition bound 2}. We begin by exhibiting numerical bounds for the functions $F(x,\chi)$ defined in equation~\eqref{eq.def of Fzchi}.

\begin{lemma} \label{dchi and echi lemma}
Assume GRH. Let $\chi_1,\dots,\chi_9$ be the nonprincipal characters modulo $q=11$, labeled as in Definition~\ref{labeling characters}. For each pair $(d(\chi),e(\chi))$ given in the table below, we have $|F(x,\chi)| \le \min\{1,d(\chi)|x|^{-e(\chi)} \}$ for all $x\in\R$.
\begin{center}$
\begin{array}{|c|r|r|r|r|}
\hline
\chi & \substack{\textstyle d(\chi)\text{ when} \\ \textstyle e(\chi)=5} & \substack{\textstyle d(\chi)\text{ when} \\ \textstyle e(\chi)=8.5} & \substack{\textstyle d(\chi)\text{ when} \\ \textstyle e(\chi)=9.5} & \substack{\textstyle d(\chi)\text{ when} \\ \textstyle e(\chi)=10} \\ \hline
\chi_1 & 820 & 1{,}855{,}630 & 21{,}021{,}079 & 73{,}516{,}699 \\ \hline
\chi_2 & 1{,}189 & 2{,}875{,}162 & 32{,}845{,}058 & 115{,}861{,}968 \\ \hline
\chi_3 & 1{,}195 & 2{,}916{,}371 & 33{,}450{,}847 & 116{,}963{,}421 \\ \hline
\chi_4 & 1{,}203 & 2{,}950{,}376 & 34{,}133{,}529 & 119{,}474{,}311 \\ \hline
\chi_5 & 678 & 1{,}549{,}125 & 17{,}785{,}195 & 61{,}586{,}977 \\ \hline
\chi_6 & 770 & 1{,}800{,}290 & 20{,}607{,}026 & 71{,}566{,}480 \\ \hline
\chi_7 & 355 & 742{,}204 & 8{,}398{,}441 & 28{,}913{,}287 \\ \hline
\chi_8 & 880 & 2{,}026{,}172 & 23{,}375{,}053 & 81{,}872{,}319 \\ \hline
\chi_9 & 715 & 1{,}590{,}237 & 18{,}149{,}703 & 63{,}453{,}141 \\ \hline
\end{array}
$\end{center}
\end{lemma}

\begin{proof}
From~\cite[equation~(2-16)]{BSRPNR} we have
\[|F(s,\chi)|\le \left(\pi |x|\right)^{-J/2}\prod\limits_{j=1}^J \left(\frac{1}{4}+\gamma_j^2\right)^{1/4}.\]
Setting $d(\chi)=\pi^{-J/2}\prod\limits_{j=1}^J \left(\frac{1}{4}+\gamma_j^2\right)^{1/4}$ and using the computed first few zeros yields the result.
\end{proof}

\begin{prop} \label{bound E2}
Assume GRH. For $q=11$ and any $2\le a\le 10$, we have $|E_2(0.2,100)| \le 2.44\times 10^{-11}$.
\end{prop}

\begin{proof}
In Proposition~\ref{proposition bound 2},
we take $b=4$ and $c=1$ and $e(\chi)=5$.
For each $2\le a\le10$,
we carefully manually choose values for $c_+$, $c_-$, $e_+(\chi)$, and $e_-(\chi)$. The values for $d_+(\chi),d_-(\chi),d(\chi)$ are then determined from $e_+(\chi),e_-(\chi),e(\chi)$
by Lemma~\ref{dchi and echi lemma}. Our choices, and the corresponding conclusions from Proposition~\ref{proposition bound 2}, are recorded in the following table.
\begin{center}$
\begin{array}{|c | c c c c | c |} 
\hline
a & c_+ & c_- & e_+(\chi) & e_-(\chi) & B_2(a,0.2,100)\le \\
\hline\hline
2 & 0.309 & -0.309 & 8.5 & 8.5 & 6.07\times 10^{-12} \\ 
\hline
3 & 0.309 & -0.809 & 8.5 & 9.5 & 1.50\times 10^{-13} \\ 
\hline
4 & 0.309 & -0.809 & 8.5 & 9.5 & 1.50\times 10^{-13} \\ 
\hline
5 & 0.309 & -0.809 & 8.5 & 9.5 & 1.67\times 10^{-13} \\ 
\hline
6 & 0.309 & -0.309 & 8.5 & 8.5 & 6.07\times 10^{-12} \\ 
\hline
7 & 0.309 & -0.309 & 8.5 & 8.5 & 4.09\times 10^{-12} \\ 
\hline
8 & 0.309 & -0.309 & 8.5 & 8.5 & 4.09\times 10^{-12} \\ 
\hline
9 & 0.309 & -0.809 & 8.5 & 9.5 & 1.67\times 10^{-13} \\ 
\hline
10 & 1 & -1 & 10 & 10 & 9.72\times 10^{-14} \\
\hline
\end{array}
$\end{center}
In all cases, $|E_2(0.2,100)|\le 4B_2(a,0.2,100) \le 2.44\times 10^{-11}$ as claimed.
\end{proof}

Finally, we work through the bound for the product truncation error $E_3(\varepsilon,C,T)$ in Proposition~\ref{E_3(varepsilon,C,T)}. Indeed, the given bound for $E_3(\varepsilon,C,T)$ is very similar in shape to the formula~\eqref{S(varepsilon,C,T)} for the approximation $S(\varepsilon,C,T)$ itself, and so it makes sense to compute them at the same time (and using similar code in SageMath).

\begin{prop}
Assume GRH. When $q=11$, we have the following values for $S(0.2,100,10^4)$ (rounded in the last decimal place) and the following upper bounds for $E_3(0.2,100,10^4)$.
\begin{center}
$\begin{array}{ |c|c|c| } 
\hline
a & S(0.2,100,10^4) & |E_3(0.2,100,10^4)|\le \\  \hline
2 &0.312963401&2.18 \times10^{-7}\\ \hline
3 &0.359656978&2.34 \times10^{-7}\\ \hline
4 &0.359656978&2.34 \times10^{-7}\\ \hline
5 &-0.030318747&2.18 \times10^{-7}\\ \hline
6 &0.312963401&2.18 \times10^{-7}\\ \hline
7 &-0.171404345&2.33 \times10^{-7}\\ \hline
8 &-0.171404345&2.33 \times10^{-7}\\ \hline
9 &-0.030318747&2.18 \times10^{-7}\\ \hline
10&0.635486213&2.79\times10^{-7}\\ \hline
\end{array}$
\end{center}
\end{prop}

\begin{proof}
We first note that the quantities ${b}_1(T,\chi)$ and $\widetilde{b}_1(T,\chi)$ in Definition~\ref{Delta_T def} can be computed from Lemma~\ref{numerical gamma sums} by subtracting the contributions from the finitely many zeros of height less than~$T$.
We choose $T=10^4$, and we obtain the following values:

\[\tilde{b_1}(10^4,\chi_1)\approx -3.42832\times 10^{-4},\quad \tilde{b_1}(10^4,\chi_2)\approx -3.42827\times 10^{-4},\quad \tilde{b_1}(10^4,\chi_3)\approx -3.42832\times 10^{-4},\]
\[\tilde{b_1}(10^4,\chi_4)\approx -3.42827\times 10^{-4},\quad {b_1}(10^4,\chi_5)\approx -1.71411\times 10^{-4}.\]

We conclude that $\widehat{b}_1(10^4) < 0.000342833$. In particular, $\widehat{b}_1(10^4)(0.2)^2 100^2 < 400\cdot 0.000342833 <1$, verifying the hypothesis of Proposition~\ref{E_3(varepsilon,C,T)}. It remains only to perform the finite calculations of both $S(0.2,100,10^4)$ from equation~\eqref{S(varepsilon,C,T)} and the upper bound for $E_3(0.2,100,10^4)$ from Proposition~\ref{E_3(varepsilon,C,T)}. The computations took approximately $4$ hours for each value of~$a$, using~SageMath 10.6 on a 2017 i7-8650U CPU.
\end{proof}

\begin{proof}[Proof of Theorem~\ref{mainthmshort}]
By Proposition~\ref{density formula prop}, for any $\varepsilon>0$ and $C\ge1$ and $T>0$ we have
\begin{align*}
\delta^{++}_a &= \frac1{4} - \frac1{4\pi^2} \lim_{\varepsilon\rightarrow 0}\int_{|\xi|>\varepsilon} \int_{|\eta|>\varepsilon}\frac{\phi_X(\xi,\eta)-\phi_X(0,\eta)\phi_X(\xi,0)}{\xi\eta}\,d\xi\,d\eta \\
&= \frac1{4} - \frac1{4\pi^2} I = \frac1{4} - \frac1{\pi^2} \bigl( S(\varepsilon,C,T) + \frac1{4}E_1(\varepsilon) + E_2(\varepsilon,C) + E_3(\varepsilon,C,T) \bigr)
\end{align*}
by equations~\eqref{equation definition bias}, \eqref{E1 def}, \eqref{SepsilonC prelim and E2}, and~\eqref{E3 def}. In particular, $\delta^{++}_a$ is approximately $\frac14-\frac1{\pi^2} S(0.2,100,10^4)$, and more precisely
\begin{align*}
\biggl| \delta^{++}_a - \Bigl( \frac14-\frac{S(0.2,100,10^4)}{\pi^2} \Bigr) \biggr| &\le \frac{|E_1(0.2)|/4 + |E_2(0.2,100)| + |E_3(0.2,100,10^4))|}{\pi^2} \\
&\le \frac{3.09\times10^{-13}/4 + 2.44\times 10^{-11} + 2.79\times10^{-7}}{\pi^2} < 2.83 \times 10^{-8}.
\end{align*}

Therefore we obtain the following table for the values of $\delta_a^{++}$ with an error of at most $4 \times10^{-8}$ (including the error from rounding off to eight decimal places):

\begin{center}
$\begin{array}{ |c|c| } 
\hline
a & \delta_a^{++} \\  \hline
2 &0.21829017 \\ \hline
3 &0.21355913 \\ \hline
4 &0.21355913 \\ \hline
5 &0.25307193 \\ \hline
6 &0.21829017 \\ \hline
7 &0.26736689 \\ \hline
8 &0.26736689 \\ \hline
9 &0.25307193 \\ \hline
10&0.18561178 \\ \hline
\end{array}$
\end{center}

Finally, by Proposition~\ref{proposition ANS}, $\delta_a^{++} + \delta_a^{+-}$ is the probability that the first coordinate of~$X$ is positive, which equals~$\frac12$ by symmetry. Similarly $\delta_a^{-+} = \frac12 - \delta_a^{++}$, and thus $\delta_a^{--} = 1 - (\delta_a^{++} + \delta_a^{+-} + \delta_a^{-+}) = \delta_a^{++}$. (The fact that these densities must add to~$1$ follows from the absolute continuity of $X$, which in turn follows from the integrability of $\phi_X$---see for example~\cite[Proposition~4.6]{MN}).
\end{proof}

\section{An explanatory model}\label{model section}

Assuming GRH and LI, the normalized error term $\Edot^\psi(x;q,a)$ has the same limiting logarithmic distribution as the distribution of the random variable
\[
2\Re \sum_{\substack{\chi\mod{q}\\\chi\neq \chi_0}} \bar{\chi}(a) \sum_{\substack{\gamma>0 \\  L(1/2+i\gamma,\chi)=0}} \frac{Z_\gamma}{1/2+i\gamma},
\]
where the $Z_\gamma$ are independent random variables each uniformly distributed on the unit circle. Individually this random variable has the same distribution as
\[
\sum_{\substack{\chi\mod{q}\\\chi\neq \chi_0}} \sum_{\substack{\gamma>0 \\  L(1/2+i\gamma,\chi)=0}} \frac{2\Re Z_\gamma}{\sqrt{1/4+\gamma^2}}
\]
since $Z_\gamma$ is invariant under rotations. Since $\Re Z_\gamma$ has variance $\frac12$, the variance of this random variable is
\[
\sum_{\substack{\chi\mod{q}\\\chi\neq \chi_0}} \sum_{\substack{\gamma>0 \\  L(1/2+i\gamma,\chi)=0}} \biggl( \frac{2}{\sqrt{1/4+\gamma^2}} \biggr)^2 \cdot \frac12 = 2 \sum_{\substack{\chi\mod{q}\\\chi\neq \chi_0}} \sum_{\substack{\gamma>0 \\  L(1/2+i\gamma,\chi)=0}} \frac1{1/4+\gamma^2}.
\]
Accounting for symmetries among the zeros of the $L(s,\chi)$, the double sum is precisely twice the sum of the numbers in Lemma~\ref{numerical gamma sums}, and thus the variance of the limiting logarithmic distribution of $\Edot^\psi(x;q,a)$ is
\[
\approx 2(0.371959 + 0.304227 + 0.817511 + 0.359942 + 0.253757) \approx 4.21479.
\]
This distribution is certainly not normal (its tails decay more quickly than those of a normal distribution, for example), but a crude model for the distribution is therefore a normal random variable with mean~$0$ and variance~$4.21479$.

For $a\not\equiv1\mod q$, the limiting logarithmic distribution of $\bigl( \Edot^\psi(x;q,1), \Edot^\psi(x;q,a) \bigr)$ is the same as the distribution of either of the random variables
\begin{align*}
& 2\Re \sum_{\substack{\chi\mod{q}\\\chi\neq \chi_0}} \bigl( 1, \bar{\chi}(a) \bigr) \sum_{\substack{\gamma>0 \\  L(1/2+i\gamma,\chi)=0}} \frac{Z_\gamma}{1/2+i\gamma} \\
& 2\Re \sum_{\substack{\chi\mod{q}\\\chi\neq \chi_0}} \bigl( 1, \bar{\chi}(a) \bigr) \sum_{\substack{\gamma>0 \\  L(1/2+i\gamma,\chi)=0}} \frac{Z_\gamma}{\sqrt{1/4+\gamma^2}}
\end{align*}
However, in this multidimensional distribution we cannot replace $\bar{\chi}(a)$ by~$1$, as we did in the one-dimensional distribution, without destroying the correlation between the coordinates of the vector $\bigl( 1, \bar{\chi}(a) \bigr) Z_\gamma$. Each coordinate of this random vector can be modeled crudely by a normal random variable with mean~$0$ and variance~$4.21479$; however, if we declared that those normal variables were independent, the two coordinates would be completely uncorrelated (for example, the probability that both were positive would simply equal~$\frac14$).

Let's turn to the $q=11$ case specifically. It turns out that one of the Dirichlet $L$-functions (mod~$11$) has a particularly low-lying zero, and we want to explore the extent to which the correlations explored in this paper are mediated by this single zero. More specifically, recall that $\chi_7$ is the Dirichlet character (mod~$11$) characterized by $\chi(2) = (e^{2\pi i/10})^7$ and thus $\chi(8) = e^{2\pi i/10}$, and let $\rho_1 \approx \frac12+i\cdot 1.23119$ be the first zero of $L(s,\chi_7)$ above the real axis. It turns out that all other zeros of Dirichlet characters (mod~$11$) above the real axis have height at least $2.47724$. The single zero~$\rho_1$ contributes $\frac2{1/4+\gamma_1^2} \approx 1.13262$ to the variance of each coordinate, which is nearly $27\%$ of the total variance~$4.21479$, and over~$3.6$ times as much as any other zero. We crudely model the contributions of all other zeros as independent normal variables of the appropriate variance; our model for $\bigl( \Edot^\psi(x;q,1), \Edot^\psi(x;q,8^k) \bigr)$ is therefore
\begin{multline} \label{one term and crude}
2\Re \biggl( \bigl( 1, \bar{\chi_7}(8^k) \bigr) \frac{Z_{\gamma_1}}{\sqrt{1/4+\gamma_1^2}} \biggr) + N(0,4.21479-1.13262) \\
\approx 1.50507\Re \bigl( (1, e^{-\pi i k/5}) Z_{\gamma_1} \bigr) + N(0,3.08218).
\end{multline}

Let's examine the probability that both coordinates of this random variable are positive. Let $F(t) = \frac12(1+\erf(x/\sqrt{2\cdot3.08218}))$ be the cumulative distribution function of the normal variable $N(0,3.08218)$. For any real number~$c$, the probability that $c+N(0,3.08218)$ is positive is $1-F(-c) = F(c)$. Therefore the probability that the right-hand side of equation~\eqref{one term and crude} lies in the first quadrant equals
\begin{equation} \label{model integral}
\frac{1}{2\pi}\int_0^{2\pi} F\bigl( 1.50507\cos(t) \bigr) F\bigl( 1.50507\cos(t-\tfrac{k\pi}{5}) \bigr) \,dt.
\end{equation}
This integral can be evaluated numerically for $1\le k\le 9$ and compared to the calculated values in Theorem~\ref{mainthmshort} (see the table and plots below). Even discarding the bottom four rows with their values duplicating the top four rows,
the relative~$\ell^2$ error is about $6.4\%$.
We see that the true densities and the differences between them are reasonably well simulated by the probabilities from this model, demonstrating that the single low-lying zero~$\rho_1$ (with its complex conjugate) is the primary cause of the cyclic ordering phenomenon. 

\begin{center}
$\begin{array}{ |c|c|c| } 
\hline
a \mod{11} & \delta_a^{++} & \rm equation~\eqref{model integral} \\  \hline
8^1\equiv8 &0.267367 & 0.289684 \\ \hline
8^2\equiv9 &0.253072 & 0.265129 \\ \hline
8^3\equiv6 &0.218290 & 0.234871 \\ \hline
8^4\equiv4 &0.213559 & 0.210316 \\ \hline
8^5\equiv10&0.185612 & 0.200890 \\ \hline
8^6\equiv3 &0.213559 & 0.210316 \\ \hline
8^7\equiv2 &0.218290 & 0.234871 \\ \hline
8^8\equiv5 &0.253072 & 0.265129 \\ \hline
8^9\equiv7 &0.267367 & 0.289684 \\ \hline
\end{array}$
\end{center}

\begin{figure}[ht]
\begin{center}
\includegraphics[width=4in]{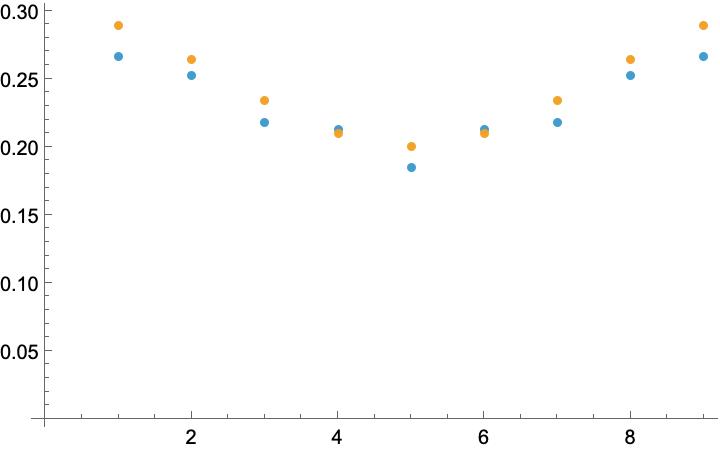}
\end{center}
\end{figure}
\pagebreak

\section*{Acknowledgments}

The authors thank Joel Feldman for a helpful discussion regarding tempered distributions. The second author was supported in part by a Natural Sciences and Engineering Research Council of Canada Discovery Grant. This work was initiated while the first author was supported by PIMS postdoctoral fellowship at the University of Lethbridge and the third author was supported by PIMS postdoctoral fellowship at the University of British Columbia.

\printbibliography
\Addresses
\end{document}